\documentclass{article}                     

\usepackage{color}
\usepackage{graphicx}

\usepackage{amssymb}
\usepackage{amsmath}
\usepackage{amsthm}
\usepackage[utf8]{inputenc}

{\newtheorem{bem}{Bemerkung}[section]}
\newtheorem{proposition}[bem]{Proposition}
\newtheorem{theorem}[bem]{Theorem}
\newtheorem{lemma}[bem]{Lemma}
\newtheorem{corollary}[bem]{Corollary}

\newtheorem{remark}[bem]{Remark}

\newtheorem{definition}[bem]{Definition}
\newtheorem{example}[bem]{Example}
\newcommand{\e} {\varepsilon}

\newcommand{\Z} {\mathbb{Z}}
\newcommand{\N} {\mathbb{N}}
\newcommand{\E} {\mathbb{E}}
\newcommand{\R} {\mathbb{R}}

\newcommand{\p} {\mathcal{P}}
\newcommand{\Q} {\mathcal{Q}}

\newcommand{\A} {\mathcal{A}}

\newcommand{\B} {\mathcal{B}}
\newcommand{\X} {\mathcal{X}}

\newcommand{\F} {\mathcal{F}}

\title{On converse Lyapunov theorems for fluid network models\thanks{This research was funded by the Volkswagen Foundation under grant I/83 087}}

\author{Michael Sch\"onlein,  Fabian Wirth\thanks{Institute for Mathematics, University of W\"urzburg, Emil-Fischer Stra\ss e. 40, 97074 W\"urzburg, Germany 
              $\{$schoenlein,wirth$\}$@mathematik.uni-wuerzburg.de} }

\begin{document}


%
%


\maketitle

\begin{abstract}
We consider the class of closed generic fluid networks (GFN) models, which
provides an abstract framework containing a wide variety of fluid
networks. Within this framework a Lyapunov method for stability of GFN
models was proposed by Ye and Chen. They proved that stability of a GFN
model is equivalent to the existence of a functional on the set of paths
that is decaying along paths. This result falls short of a converse Lyapunov
theorem in that no state dependent Lyapunov function is constructed. In this
paper we construct state-dependent Lyapunov functions in contrast to
path-wise functionals. We first show by counterexamples that closed GFN
models do not provide sufficient information that allow for a converse
Lyapunov theorem. To resolve this problem we introduce the class of strict
GFN models by forcing the closed GFN model to satisfy a concatenation
and a semicontinuity condition of the set of paths in dependence
of initial condition. For the class of strict GFN models we define a state-dependent Lyapunov and show that a converse Lyapunov theorem holds. Finally, it is shown that common fluid network models, like general work-conserving and priority fluid network models as well as certain linear Skorokhod problems define strict GFN models.

\end{abstract}

\section{Introduction}\label{intro}

An effective tool to model complex manufacturing systems, computer systems or
telecommunication networks is the family of multiclass queueing networks. An
example for this occurs in semiconductor fabrication, where production lines are
modeled as reentrant lines, which are a special case of multiclass queueing
networks. Especially in the pursuit of deriving good control strategies for
multiclass queueing networks the question of stability arises. For a long
period a common belief was that a sufficient condition for stability is that
the traffic intensity is strictly less then one. But in 1993 Kumar and Seidman
\cite{kumarseidman90} presented a network with two stations processing four
types of jobs which is unstable although the traffic intensity at each station
is less than one. This example inspired a number of examples with different
service disciplines, like first-in-first-out (FIFO) and priority, that have
surprising properties. In the literature they are known as the Lu-Kumar
network, the Rybko-Stolyar network or the Bramson network, see
e.g. \cite{bramsonfifo} or \cite{bramsonLN}, \cite{daihavv04} and
\cite{rybkostolyar}. In recent years further disciplines like maximum pressure
and join-the-shortest-queue are investigated \cite{daihaba07}, \cite{dailin05}, \cite{dailin08}. Rybko and Stolyar \cite{rybkostolyar} and Dai \cite{dai}
pursued the strategy of rescaling the stochastic processes that describe the
dynamics of a multiclass queueing network and considered the limit obtained under scaling. This limit is called the fluid limit model for the queueing network
and is a continuous deterministic model. Of course, deterministic models
are much easier to investigate. The great benefit of this approach is, that
the stability of the corresponding fluid limit model is sufficient for the
stability of a multiclass queueing network \cite{dai}. In addition, there are conditions for instability of queueing networks relative to their fluid limit model \cite{dai96,pukhalski2000nonergodicity}. A discussion of the
relationship between queueing networks and fluid models can be found in
\cite{bramsonLN}.

Due to this fact the question arises, under which conditions fluid limit
models are stable. A fluid model is called stable if the fluid level process
$Q$ with unit initial level is drained to zero in a uniform finite time $\tau$ and
remains zero beyond $\tau$. Of course, conditions that guarantee stability
depend on the service discipline of the network. In \cite{chen} Chen states
necessary and sufficient conditions for stability of general work-conserving
fluid networks. Stability conditions for fluid networks under FIFO and
priority discipline have been derived by Chen and Zhang \cite{chenfifo},
\cite{chenpriority}. Often the strategy to prove such conditions is to use a
Lyapunov function. In this context a locally Lipschitz function $V:\mathbb{R}_+^K
\rightarrow \mathbb{R}_+$ such that $V(x)=0$ if and only if $x=0$ is called a
Lyapunov function, if there exists a constant $\varepsilon>0$ such that for
each fluid model solution it holds that
\begin{equation*} 
 \tfrac{d}{dt}  V(Q(t)) \leq - \varepsilon
\end{equation*}
whenever $Q(t)\not=0$ and the derivative at time $t$ exists for 
the map $s\mapsto V(Q(s))$. 
For more details
see \cite{daischool}. Within this framework linear Lyapunov functions of the
form
\begin{equation*} 
  V(x) = h^Tx, \qquad x \in \R_+^n 
\end{equation*}
where $h$ is some positive vector in $\mathbb{R}_+^K$ are used to establish a sufficient condition for the stability of fluid network models under a priority discipline \cite{chenpriority}. The special case for this where $h=(1,...,1)^T$ is used the show that a fluid model of a re-entrant line operating under last-buffer-first-served (LBFS) service discipline is stable, if the usual traffic condition $\rho_j< 1$ is satisfied for all stations $j$ \cite{daischool}. This special case is also used to prove a stability condition for fluid networks under the join-the-shortest-queue discipline \cite{daihaba07}. Ye and Chen investigated fluid networks under priority disciplines by using piecewise linear Lyapunov functions of the form
\begin{equation*} 
  V(x) = \max_{1\leq j  \leq N}  h_j^T x
\end{equation*}
for some nonnegative vectors $h_1,...,h_N$, for details see \cite{yechenpw}. This approach yields a sharper stability condition for fluid networks under priority discipline than in \cite{chenpriority}. Furthermore, in the verification of a stability condition for fluid networks under general work-conserving disciplines a quadratic Lyapunov function
\begin{equation*} 
  V(x) = x^T\,A\, x
\end{equation*}
with a strictly copositive matrix $A$ is used \cite{chen}. What all the works mentioned above have in common is that the existence of Lyapunov functions is only shown to be sufficient for stability. 

Before we investigate the question whether the existence of a Lyapunov function is also necessary for the stability of a fluid network, we recall briefly the basic idea of a Lyapunov function from the theory of dynamical systems. For a detailed description the reader is referred e.g. to \cite{bacciotti}, \cite{hpI}.  
Consider a dynamical system
\begin{equation}\label{tvsys}
 \dot x = f(x), \qquad x \in \mathbb{R}^n, \, t \in [0, \infty)
\end{equation}
with initial condition $x(0)=x_0$ and continuous $f$, where the origin is an equilibrium position, i.e. $f(0)=0$. A real valued map $V: B_r \subset \mathbb{R}^n \rightarrow \mathbb{R}$ is called a strict Lyapunov function for \eqref{tvsys} if (i) it is positive definite and proper, i.e. there exist continuous and strictly increasing function $a,b:[0,r)\rightarrow [0,\infty)$ with $a(0)=b(0)=0$ such that
\begin{equation}\label{tvlyap-1}
 a(\|x\|) \leq V(x) \leq b(\|x\|),\qquad x \in B_r 
\end{equation}
and (ii) if there exists a continuous and strictly increasing function $w:[0,r)\rightarrow [0,\infty)$ with $w(0)=0$ such that for every solution $x(\cdot)$ and each interval $I\subset [0,\infty)$ one has
\begin{equation}\label{tvlyap-2}
 V(x(t_2))- V(x(t_1)) \leq - \int_{t_1}^{t_2} w(\|x(t)\|) dt 
\end{equation}
for each $t_1<t_2 \in I$ provided that $x(\cdot)$ is defined on $I$ and $x(t)\in B_r$ for $t\in I$. It is well known that the origin is 
locally asymptotically stable, if and only if there is a strict Lyapunov function \cite{bacciotti}.

In order to obtain a so called converse Lyapunov theorem for fluid networks Ye and Chen followed a different, more general approach \cite{Yechen}. They collected the characteristic properties of fluid networks and defined a generic fluid network (GFN) model as set $\Phi$ of functions $Q:\mathbb{R}_+ \rightarrow \mathbb{R}_+^K$ that satisfy a few natural properties. A precise description of a GFN model is given in Section~\ref{GFN}. They proved that stability of a GFN model is equivalent to the property that for every function $Q \in \Phi$ a Lyapunov functional $v:\mathbb{R}_+\rightarrow \mathbb{R}_+$ is decaying along $Q$. In particular, $v$ can be chosen as
\begin{equation}\label{tvlyap}
 v(t) = \int_{t}^{\infty} \|Q(s)\| ds. 
\end{equation}

It can be seen that this approach differs from the one taken in the theory of dynamical systems in which Lyapunov functions are state-dependent. The dependence on solutions is undesirable, because the benefit of Lyapunov's second method is that trajectories need not be known to be able to determine stability, whereas the method of Ye and Chen requires the knowledge of all solutions. In this paper we define a state-dependent Lyapunov function and prove a converse Lyapunov theorem. 

This paper is organized as follows. In the Section~\ref{GFN} we recall the
definition of a GFN model from \cite{Yechen}. Further we discuss
counterexamples to emphasize that the class of (closed) GFN models is too
general to provide a converse Lyapunov theorem with state-dependent Lyapunov
functions. In the Section~\ref{results} we introduce the class of strict GFN models by forcing the closed GFN models to satisfy additionally a
concatenation and a lower semicontinuity property. The concatenation property
is essential for state-dependent Lyapunov functions whereas lower
semicontinuity gives the additional benefit of continuity. For this model
class we define a state-dependent Lyapunov function and prove that within this
framework the stability of a strict GFN model is equivalently
characterized by the existence of a state-dependent Lyapunov function. In
Section~\ref{di-elements} we recall some results from differential inclusions
and viability theory that will be useful in Section~\ref{appl-fn}. There we
show that 
general work-conserving and priority fluid networks define strict GFN
models. In Section~\ref{sec:fluidlimit} we consider fluid limit models of
queueing networks for a specific class of disciplines and in
Section~\ref{sec:skorokhod} we comment on linear Skorokhod problems. 
In Section~\ref{FIFOsec} we explain why the approach of the current paper is
not immediately applicable to FIFO systems. We conclude in
Section~\ref{conclusions}. 

We now collect some notations that will be used throughout the paper. By $\R_+^K$ we denote the nonnegative orthant $\{x \in \R^K : x \geq 0\}$, where $\geq$ has to be understood component-wise. Throughout the paper we mostly consider the space $(\R_+^K,\|\cdot\|)$ with $\|x\|:= \sum_{i=1}^K |x_i|$. 
Let $D(\R_+,\R_+^K)$ denote the space of right continuous functions $f :\R_+ \rightarrow \R_+^K$ having left limits that is endowed with the Skorokhod topology \cite{ethier1986markov}. Let $C(\R_+,\R_+^K)$ be the subset of continuous functions.
A sequence of functions, denoted by $(f_n(t))_{n\in \mathbb{N}}$, in $D(\R_+,\R_+^K)$ is said to converge uniformly on compact sets (u.o.c.) to a continuous function $f(t)\in C(\R_+,\R_+^K)$, if for any $T>0$
\begin{equation*}
 \lim_{n\rightarrow  \infty} \, \, \sup_{t\in [0,T]} \, \|f_n(t) - f(t)\|=0.
\end{equation*}
We say that a function $g:\R_+^K \rightarrow \R$ is upper semicontinuous in $a\in \R_+^K$, if $g(a) \geq \limsup_{x \rightarrow a} g(x)$. Of course, $g$ is called upper semicontinuous if it is upper semicontinuous for every $a \in \R_+^K$. Further a function $g : \R_+^K \rightarrow \R_+$ is lower semicontinuous at $a \in  \R_+^K$ if $-g$ is upper semicontinuous in $a\in \R_+^K$ and $g$ is called lower semicontinuous if $g$ is lower semicontinuous in every point. We use $\rightsquigarrow$ to denote set-valued maps. Let $X$ and $Y$ denote metric spaces. A set-valued map $F:X \rightsquigarrow Y$ is a mapping that maps every $x \in X$ into a set $F(x)$, called the value of $F$ at $x$. The domain of a set-valued map $F:X \rightsquigarrow Y$ is the subset of elements $x\in X$ such that the values $F(x)$ are non empty. i.e. $\mbox{dom}(F)=\{x\in X\, :\, F(x)\not =\emptyset\}$. The image of $F$ is the union of all values $F(x)$ for all $x \in X$. The graph of a set-valued map $F$ is $\mbox{graph}(F):=\{(x,y)\in X\times Y\, : \,y \in F(x)\}$. A set-valued map $F$ is said to be closed-valued if the values of $F$ are closed, i.e. for every $x\in X$ the set $F(x)$ is closed. Accordingly, $F$ is said to be convex if the images are convex. Moreover, a set-valued map $F:X\rightsquigarrow Y$ is called lower semicontinuous at $x \in \mbox{dom}(F)$ if for any $y \in F(x)$ and for any sequence of elements $(x_n)_{n \in \N} \in \mbox{dom}(F)$ converging to $x$, there exists a sequence $(y_n)_{n \in \N}$ with $y_n \in F(x_n)$ converging to $y$. $F$ is said to be lower semicontinuous if it is lower semicontinuous at every point $x \in \mbox{dom}(F)$.
In addition, 
a set-valued map
$F$ is called upper semicontinuous at $x \in$ dom$(F)$,
if for any open neighborhood $U \supset F(x)$ there is an $\e>0$ such that
for all $x' \in B(x,\e)\cap {\mathrm dom}\,(F)$ it holds that $F(x')
\subset U$.  Again $F$ is said to be upper semicontinuous if it is upper
semicontinuous at every point $x \in \mbox{dom}(F)$.

Finally, by $\mathcal{K}$ we denote the set of continuous functions $w:\R_+ \rightarrow \R_+$ that satisfy $w(0)=0$ and are strictly increasing.

\section{Generic fluid network models}\label{GFN}

In this section we consider generic fluid network models introduced by Ye
and Chen in 
\cite{Yechen}.  They present a trajectory-based Lyapunov method for
characterizing the stability of fluid networks, in which the Lyapunov
function depends on the path of the closed GFN model.  First we recall
from \cite{Yechen} the definition of a closed generic fluid network
(closed GFN) model and the conditions for a function to be a Lyapunov
function. Then we define a candidate for a Lyapunov function that does not
depend on the path and show that in the setting it is not continuous in
general. 
Further we give a counterexample that shows that within the class
of closed GFN models the concatenation of two paths is not automatically
contained in a closed GFN model, if the queue lengths at the time of
concatenation coincide.

\begin{definition}\label{gfn}\cite{Yechen}
A nonempty set $\Phi$ of functions $Q(\cdot) :\R_+ \rightarrow \R_+^K$ is said to be a \emph{GFN model}, if
\begin{enumerate}
\item[(a)] there is a $L>0$, such that for any $Q(\cdot) \in \Phi$ and $t,s \in  \R_+$ it holds that 
\begin{equation*}
\|Q(t)-Q(s)\| \leq L \, |t-s|.
\end{equation*}
\item[(b)] $Q(\cdot) \in \Phi$ implies $\frac{1}{r}Q(r\cdot) \in \Phi$ for all $r>0$.
\item[(c)] $Q(\cdot) \in \Phi$ implies $Q(s+\cdot) \in \Phi$ for all $s\geq0$.
\end{enumerate}
Furthermore, if the following condition is also satisfied, then we call $\Phi$ a \emph{closed GFN model}.
\begin{enumerate}
\item[(d)] if a sequence $(Q_n)_{n \in \N} \subset \Phi$ converges to $Q_*$ u.o.c, then $Q_* \in \Phi$.
\end{enumerate}
\end{definition}

Any element $Q(\cdot)$ of $\Phi$ is called a path (of $\Phi$) and the set of paths with initial level one is denoted by $\Phi(1) = \{ Q(\cdot) \in \Phi: \|Q(0)\|=1\}$. 
Condition (a) states that the functions $Q(\cdot)$ are Lipschitz continuous, where condition (b) is a scaling property and condition (c) is a shift property. 
We note that the terminology closed is not related to closed queueing networks where a fixed number of job circulate in the network. 
Rather, the content of condition (d) is that the set $\Phi$ is closed in the topology of uniform convergence on compact sets. For future use we also introduce for $x \in \R_+^K$ the set $ \Phi_x = \{ Q (\cdot)\in \Phi: Q(0)=x\}$. Moreover we 
recall from \cite{Yechen} the definition of stability of a GFN model. 

\begin{definition}\label{GFNstab}
A GFN model $\Phi$ is said to be \emph{stable}, if there exists a $\tau >0$, such that $Q(\tau + \cdot) \equiv 0$ for any path $Q(\cdot) \in \Phi(1)$.
\end{definition}

From a dynamical systems perspective the definition of stability for a closed GFN model $\Phi$ seems to deviate from the asymptotic stability in the Lyapunov sense, where asymptotic stability is described by the following. The zero path is said to be asymptotically stable, provided that

\begin{lemma}\label{stabdefequiv}
$\Phi$ is stable if and only if the zero path is asymptotically stable in the sense of Lyapunov.
\end{lemma}

\begin{proof}
Suppose that $\Phi$ is stable, i.e. there is a $\tau < \infty$ such that $Q(t)=0$ for all $Q \in \Phi(1)$ and all $t\geq \tau$. Let $Q(\cdot) \in \Phi$ and denote $q:=\|Q(0)\|$. The scaling property implies that $\frac{1}{q} Q(q\,t) \in \Phi(1)$ and consequently, $Q(qt)=0$ for all $t\geq q^{-1}\tau$. This implies $\lim_{t \rightarrow \infty} \|Q(t)\| =0$ and attractivity holds true. In \cite{stolyar95} Theorem~6.1 it is shown that the stability of $\Phi$ is equivalent to the condition
\begin{align*}
 \inf_{t\geq 0} \|Q(t)\| \leq 1
\end{align*}
for all $Q(\cdot) \in \Phi(1)$. This means that there is no $Q(\cdot) \in \Phi$ such that $\|Q(t)\| \geq \|Q(0)\|$ for all $t\geq 0$. Stability then follows by choosing $\e =\delta$.\\
Conversely, let $Q(\cdot)\equiv 0$ be asymptotically stable in the sense of Lyapunov. Due to the scaling property it suffices to consider a $Q(\cdot) \in \Phi(1)$. Then, by attractivity it holds that $\lim_{t \rightarrow \infty} \|Q(t)\| =0$. Proceeding exactly as in the proof of Theorem~6.1 in \cite{stolyar95} shows the assertion.

\end{proof}

The notion of stability of a GFN may also be 
expressed by saying that the zero fluid level process $Q_0(\cdot) \equiv
0$ is the unique stable and attractive fixed point of the shift
operator $\delta_{\tau}Q\,(\cdot):=Q(\cdot+\tau)$ defined on $\Phi$.

The Lyapunov method to characterize stability of closed GFN models presented in \cite{Yechen} is as follows. A GFN model $\Phi$ is said to satisfy the L-condition, if there exist class $\mathcal{K}$-functions $w_i: \R_+ \rightarrow \R_+$, $i=1,2,3$ such that for any GFN path $Q\in \Phi$ there exists an absolutely continuous function $v: \R_+ \rightarrow \R_+$ such that
\begin{align}\label{fLF1-functionals}
w_1(\|Q(t)\|) &\leq v(t)  \leq w_2(\|Q(t)\|),\\ \label{fLF2-functionals}
\dot v(t) &\leq -  w_3(\|Q(t)\|)
\end{align} 
for almost all $t\geq 0$. The corresponding converse Lyapunov theorem is then

\begin{theorem}
A GFN model $\Phi$ is stable if and only if the $L$-condition is satisfied. In particular, given $Q \in \Phi$ the function $v$ can be chosen as
\begin{equation}\label{lyapfuncyechen}
  v(t) :=  \int_t^{\infty} \| Q(s)\| ds.
\end{equation} 
\end{theorem}

We note that an equivalent way of interpreting $v$ is as a functional $\bar v:\Phi\rightarrow \R_+$ on the GFN model with the following properties. There are comparison function such that for each path $Q(\cdot) \in \Phi$ its value under the functional $\bar v$ can be estimated from below and above by its initial value. That is, for any $Q(\cdot) \in \Phi$ it holds that
\begin{align*}
w_1(\|Q(0)\|) \leq \bar v(Q)  \leq w_2(\|Q(0)\|).
\end{align*}
Furthermore, the evolution of $\bar v(Q)$ can also be estimated in terms of a comparison function. Precisely, the mapping $t\mapsto \bar v(Q)(t+ \cdot)$ satisfies
\begin{align*}
\frac{d}{dt} \bar v(Q)(t+\cdot) \leq -  w_3(\|Q(t)\|).
\end{align*}
For this reason we refer to $v$, interpreted as $v(0) =: \bar v(Q)$, as a Lyapunov functional. As mentioned in the introduction the drawback of this definition is that the Lyapunov functional is path-dependent as opposed to be state-dependent, which is the basic idea of a Lyapunov function for a dynamical system. 
The definition of a Lyapunov function that only uses information of the
state is as follows. We denote by $ \A (\Phi)=\{x \in \R_+^K : \exists \,
Q(\cdot) \in \Phi,\, Q(0)=x\}$.

\begin{definition}
\label{def:Lyapdef}
Given a GFN model $\Phi$ a function $V : \A(\Phi) \rightarrow \R_+$ is said
to be a Lyapunov function, if there exist class $\mathcal{K}$ functions
$w_i : \R_+ \rightarrow \R_+,\, i=1,2,3$ such that
\begin{align}\label{fLF1}
w_1(\|x\|) \leq V(x)  &\leq w_2(\|x\|)\,,\quad x\in {\cal A}(\Phi)\\\label{fLF2}
V(Q(t)) - V(Q(s)) &\leq - \int_{s}^{t} \,w_3(\|Q(r)\|)\,dr
\end{align}
for all $0\leq s \leq t \in \R_+$ and all paths $Q(\cdot) \in \Phi$.
\end{definition}

For our purposes a certain candidate is useful; we consider in particular
$V : \A(\Phi) \rightarrow \R_+ \cup \{\infty\}$ defined by

\begin{equation}\label{lyapfunc}
V(x) = \sup_{Q (\cdot) \in \Phi_x} \int_0^{\infty} \| Q(s)\| ds.
\end{equation}

In the sequel we assume that $\A(\Phi) = \R_+^K$. The function $V$ defined
in \eqref{lyapfunc} can be interpreted as a measurement of the state $x$
in the sense that $V(x)$ represents the total possible fluid mass that the
network has to deal with. An interesting question concerns the regularity
of $V$. Of course, we aim for continuous dependence on the state, as this
would entail robustness of stability, see \cite{teel2000smooth},
\cite{kelletteel04}. Note that for stable closed GFN models the supremum
in \eqref{lyapfunc} is actually attained because of the requirement of
closedness in Definition~\ref{gfn}~(d).

\begin{proposition}\label{usc}
  If $\Phi$ is a stable closed GFN model, then the function $V: \A(\Phi)
  \rightarrow \R_+ $ defined in \eqref{lyapfunc} is well defined and upper
  semicontinuous.
\end{proposition}

\begin{proof}
It is an easy consequence of Lipschitz continuity, scaling property and stability that $V(x)$ as defined in \eqref{lyapfunc} is finite. Let $x \in \R_+^K$ and $(x_n)_{n \in \N} \subset \R_+^K$ be a sequence that converges to $x$. As $\Phi$ is stable the set $\{ V(x_n): n \in \mathbb{N}\}$ is bounded. Hence there exists a subsequence $(x_{n_l})_{l\in \mathbb{N}}$  such that 
\begin{equation*}
\limsup_{n\rightarrow \infty} V(x_n) = \lim_{l\rightarrow \infty} V(x_{n_l})= \lim_{l\rightarrow \infty}  \int_0^{\infty} \| Q_{n_l}(s)\| ds
\end{equation*}
with $Q_{n_l}(0)=x_{n_l}$. Now, consider the family $\{ Q_{n_l}(\cdot) : l \in \mathbb{N} \,\}$. Since $\Phi$ is stable the family $\{ Q_{n_l}(\cdot) : l \in \mathbb{N} \,\}$ is bounded. By condition (a) in Definition~\ref{gfn} there is a single Lipschitz constant for any path $Q_{n_l}(\cdot)$ of the family $\{ Q_{n_l}(\cdot) : l \in \mathbb{N} \,\}$ and thus the family is equicontinuous. By the theorem of Arzel\`{a}-Ascoli there exists a subsequence which converges u.o.c. to some $Q_*(\cdot)$ with $Q_*(0)=x$. Since the model is closed it follows that $Q_*(\cdot) \in \Phi$. Hence by the definition of $V$ it holds that 
\begin{align*}
\limsup_{n \rightarrow \infty} V(x_n)  = \lim_{l \rightarrow \infty} \int_0^{\infty} \|Q_{n_l}(s)\|ds= \int_0^{\infty}\|Q_*(s)\|ds 
\leq V(x).
\end{align*}
This shows the assertion. 
\end{proof}

As we are interested in the continuity of $V$ the question remains whether
$V$ is also lower semicontinuous.

\begin{example}
Let $K=2$ and 
\begin{equation*}
 \Phi = \left\{  
\begin{pmatrix}
(x_1 - t)^+\\
(x_2 - t)^+\\
\end{pmatrix},\, 
\begin{pmatrix}
(c - \frac{1}{2}t)^+\\
(c - \frac{1}{2}t)^+\\
\end{pmatrix}: x_1,x_2, c \in \mathbb{R}_+ \right\}.
\end{equation*}
It is easy to check that $\Phi$ is a stable closed GFN model. We consider $x_0= (1 \, \,1)^T$ and $x_n=(1+ \frac{1}{n}\, \, 1- \frac{1}{n})^T$. It holds that
\begin{align*}
\lim_{n \rightarrow \infty} V(x_n) &=
 \lim_{n \rightarrow \infty} \int_0^{\infty} (1+\frac{1}{n} -t)^+ + (1-\frac{1}{n} -t)^+ dt\\
&= \lim_{n \rightarrow \infty}\frac{1}{2}\left( (1+\frac{1}{n})^2 + (1-\frac{1}{n})^2 \right) 
=1 \\&
< 2= \int_0^2 2(1-\frac{1}{2}t) dt=  V(x_0).
\end{align*}
So $V$ defined by (\ref{lyapfunc}) is not necessarily lower semicontinuous for stable closed GFN models.
\end{example}

\begin{remark}\label{rem1}
The example shows that in the frame of Definition~\ref{gfn} our candidate $V$ is not continuous in general. The problem with this example is that along the diagonal a particular solution exists which is not approximated by solutions starting close to but not on the diagonal. 
\end{remark}

The key property of a Lyapunov function $V$ for a dynamical system is, that $V$ is decreasing along trajectories. The trajectories in the context of closed GFN models are the paths. The next example addresses this problem. Here the concatenation of paths plays a key role. For this reason we provide a definition.
\begin{definition}
Let $\Phi$ be a closed GFN model and suppose that $Q_1(\cdot),Q_2(\cdot)$ are paths of $\Phi$ such that for some $t^* \geq 0$ it holds that $Q_1(t^*)=Q_2(0)$. Then $Q_1\diamond_{t^*}Q_2$ defined by 
\begin{equation*}
Q_1\diamond_{t^*}Q_2(t):= \begin{cases} 
Q_1(t)     &\quad \mbox{ for } \quad 0\leq t \leq t^*, \\
Q_2(t-t^*) &\quad \mbox{ for } \quad \quad \quad t \geq t^*
\end{cases}
\end{equation*}
is called the concatenation of $Q_1(\cdot)$ and $Q_2(\cdot)$ at $t^*$.
\end{definition}

In the previous definition note that if $Q_1(t^*)=Q_2(s)$ for an arbitrary
$s\geq0$, then because of the shift property we can consider the
concatenation of $Q_1(\cdot), Q_2(s+\cdot)$. In this sense evaluation of
$Q_2$ at $0$ in the definition poses no restriction.

\begin{example}
\label{example2}
 Let $K=2$ and define for given $x_1,x_2 \in \mathbb{R}_+$ the paths
\begin{equation*}
Q_1 (t)= \begin{cases}
\begin{pmatrix} x_1 -t \\ x_2 +t\end{pmatrix}  & \quad  \text{if } \quad 0\leq t \leq x_1, \\
\begin{pmatrix} 0 \\ x_1 + x_2 -t\end{pmatrix}^+ & \quad  \text{if } \quad t \geq x_1.
\end{cases}
\end{equation*}
and
\begin{equation*}
Q_2(t)= \begin{cases} 
\begin{pmatrix} x_1 + t\\ x_2 -t \end{pmatrix}  & \quad  \text{if } \quad t \leq x_2, \\
\begin{pmatrix} x_1 + x_2 -t \\  0 \end{pmatrix}^+  & \quad  \text{if }\quad t \geq x_2.
\end{cases}
\end{equation*}
Then consider the stable closed GFN model
\begin{equation*}
\Phi = \left\{ Q_1(\cdot), Q_2(\cdot)\, :\, x_1,x_2 \in \mathbb{R}_+ \right\}.
\end{equation*}
In this GFN model it is obvious that paths cannot be concatenated. However, let us assume that $V$ is a state-dependent Lyapunov function which is decaying along paths. The closed GFN model $\Phi$ has the following property. For every state $z=(z_1,z_2)$ there is a state $y=(y_1,y_2)$ such that there two paths that go to zero, where one path starts in $z$ and passes through $y$ and the other path starts in $y$ and passes through $z$. As $V$ is decaying along paths it follows that
\begin{align*}
 V(z) < V(y) \quad \text{ and } \quad V(y)<V(z),
\end{align*}
which is a contradiction.
\end{example}

    Example~\ref{example2} shows that in the framework of
    Definition~\ref{gfn} there are GFNs that are stable and for which no
    Lyapunov function in the sense of Definition~\ref{def:Lyapdef} can be
    defined.
    It will thus be the aim of the following section to identify
    situations where this is possible.

\section{A Converse Lyapunov Theorem}\label{results}

In this section we present a way out of the dilemma. We 
restrict the class of closed GFN models by adding two conditions, namely a concatenation property and a lower semicontinuity property. 
Fluid models with these properties are called strict GFN model. 
The
main result of this section is that the Lyapunov function candidate
\eqref{lyapfunc} is appropriate to prove a converse Lyapunov theorem for
the class of strict GFN models. The road map is as follows. First we
present the two additional conditions for the closed GFN model. 
After that we show that
under this conditions the candidate \eqref{lyapfunc} is continuous. In the
sequel we prove the main theorem. Similar to the closed GFN model we
introduce the following notations $\Q(1) = \{ Q(\cdot) \in \Q:
\|Q(0)\|=1\}$ and $\Q_x = \{ Q (\cdot)\in \Q: Q(0)=x\}$ for $x \in
\R_+^K$.

\begin{definition}\label{mGFN}
A set $\mathcal{Q}$ of functions $Q(\cdot) :\R_+ \rightarrow \R_+^K$ is  a \emph{strict GFN model}, if
\begin{enumerate}
\item[(\textit{a}')] it is a closed GFN model
\item[(e)] for GFN paths ${Q_1}(\cdot),{Q_2}(\cdot) \in \mathcal{Q}$ with
  $Q_1(t^*)=Q_2(0)$ for some $t^*\in \R_+$, the concatenation
  $Q_1\diamond_{t^*}Q_2(\cdot)$ is also a path of $\mathcal{Q}$.
\item[(f)] there is a $T>0$ such that the set-valued map $x \rightsquigarrow \Q_{x} \big\vert_{[0,T]}$ is lower semicontinuous. 
\end{enumerate}
\end{definition}
It is possible that a closed GFN model satisfies (e) and not (f). We do
not introduce yet another name for such GFN models but simply speak of a
closed GFN model satisfying (e).

\begin{remark}
    \label{semicontremark}
    We note for further reference, that for GFN models satisfying (e) the
    semicontinuity condition (f) can be stated equivalently as
    \begin{enumerate}
      \item[(f')] for each $x_0$ so that $\Q_{x_0}\neq \emptyset$ there
        exists a $T(x_0)>0$ such that the set-valued map $x
        \rightsquigarrow \Q_{x} \big\vert_{[0,T(x_0)]}$ is lower
        semicontinuous at $x_0$.
    \end{enumerate}
    It is clear that (f) implies (f'). Conversely, note first that the
    uniform Lipschitz constant guaranteed by Definition~\ref{gfn}~(a)
    implies that if a sequence of paths $(Q_n)_{n \in \N}$ converges
    u.o.c. on an interval $[0,T_1), T_1<\infty$, then the sequence
    converges uniformly on the closed interval $[0,T_1]$. Now fix any
    $T>0$, $x_0$ and a $T_0:=T(x_0)$ such that (f') holds. Choose
    $Q(\cdot) \in\Q_{x_0}$ and a sequence $(x_n)_{n \in \mathbb{N}}$ converging to $x_0$. We have to
    construct a sequence $Q_n(\cdot)\in \Q_{x_n}$ such that $Q_n(\cdot)\to
    Q(\cdot)$ uniformly on $[0,T]$. We may assume that $T_0 < T$ as
    otherwise there is nothing to show. By assumption there exist
    $Q^1_n(\cdot) \in \Q_{x_n}$ such that $Q_n^1(\cdot)\to Q(\cdot)$
    uniformly on $[0,T_0]$. In particular, $Q_n^1(T_0)\to Q(T_0)$.  By the
    shift property $Q(T_0+\cdot)\in \Q_{Q(T_0)}$ and so for $T_1:=
    T(Q(T_0))$ we may by (f') choose a sequence $\tilde{Q}_n^1(\cdot) \in
    \Q_{Q_n^1(T_0)}$ such that $\tilde{Q}_n^1(\cdot)\to Q(T_0+\cdot)$
    uniformly on $[0,T_1]$.

    Now define the concatenation $Q^2_n:=
    Q_n^1\diamond_{T_0}\tilde{Q}_n^1(\cdot)$ and note that
    $Q_n^2(\cdot)\to Q(\cdot)$ uniformly on $[0,T_0+T_1]$. Repeating this
    step countably often, we can construct an open interval $[0,\bar{T})$
    such that there exist $\bar{Q}_n(\cdot)\in\Q_{x_n}$ such that
    $Q_n(\cdot)\to Q(\cdot)$ u.o.c. on $[0,\bar{T})$. Assume that
    $\bar{T}<\infty$ is chosen as the maximal real for which this
    u.o.c. convergence is possible.

    Then by our first remark $Q_n(\cdot) \to Q(\cdot)$ uniformly on
    $[0,\bar{T}]$. Then we can repeat the argument and extend the uniform
    convergence to the interval $[0,\bar{T}+T(Q(\bar{T}))]$. This
    contradicts the assumption that $\bar{T}$ was chosen to be
    maximal. This shows the equivalence, as $\bar{T}$ can be arbitrarily
    large and so chosen to be bigger than $T$. 
\end{remark}

We have seen that the absence to certain concatenations is an
impediment to the existence of Lyapunov function in Example~\ref{example2}.
Next we show that conditions (e) and (f) close the
gap from upper semicontinuity to continuity.

\begin{proposition}\label{Vcont}
If $\mathcal{Q}$ is a stable strict GFN model, then $V$ defined in
\eqref{lyapfunc} is continuous.
\end{proposition}

\begin{proof}
We show that $V$ is lower semicontinuous as the continuity of $V$ 
then follows
together with 
Proposition~\ref{usc}. Let $x_*\in \R_+^K$ and $Q_*(\cdot) \in \Q_{x_*}$ be such that
\[V(x_*)= \int_0^{\infty} \|Q_*(s)\|ds.\]
Further let $(x_n)_{n \in \N}$ be a sequence that converges to $x_*$. By
condition (f) in Definition~\ref{mGFN} there exists a $T>0$ and a sequence
$\left(Q_n(\cdot)\big\vert_{[0,T]}\right)_{n \in \N}$ in $\Q_{x_n}
\big\vert_{[0,T]}$ that converges 
uniformly to $Q_*(\cdot)\big\vert_{[0,T]}$.
In particular, $x^1_n:= Q_{n}(T) \big\vert_{[0,T]}$ converges to
$x^1:= Q_{*}(T) \big\vert_{[0,T]}$ as $n \rightarrow \infty$. Moreover,
for $Q^1_{*}(\cdot) \big\vert_{[0,T]} \in \Q_{x^1}\big\vert_{[0,T]}$ such
that $Q^1_{*}(\cdot) \big\vert_{[0,T]} =Q_{*}(\cdot) \big\vert_{[T,2T]}$
condition~(f) yields the existence of a sequence $Q^1_n(\cdot) \in
\Q_{x^1_n} \big\vert_{[0,T]}$ satisfying
\begin{equation*}
\lim_{n \rightarrow \infty} Q^1_{n}(\cdot) \big\vert_{[0,T]} = Q^1_{*}(\cdot) \big\vert_{[0,T]}\quad\mbox{ uniformly\,. } 
\end{equation*}
Using the concatenation property~(e) we have a sequence $(Q_n(\cdot)|_{[0,2T]})_{n \in \N} \in \Q_{x_n}|_{[0,2T]}$ that converges u.o.c. to $Q_*(\cdot)|_{[0,2T]} \in \Q_{x_*}|_{[0,2T]}$.
A successive continuation in this manner yields the existence of a sequence $Q_n(\cdot) \in \Q_{x_n}$ that converges u.o.c. to $Q_*(\cdot) \in \Q_{x_*}$. 
As $\mathcal{Q}$ is stable and using the same arguments as in the proof of Proposition~\ref{usc} we have
\begin{align*}
V(x_*) = \int_0^{\infty} \| Q_*(s)\| ds  = \lim_{n \rightarrow \infty} \int_0^{\infty} \| Q_n(s)\| ds  
\leq\liminf_{n \rightarrow \infty} V(x_n).
\end{align*}
That is, $V$ is lower semicontinuous.
\end{proof}

Now we state the main theorem.

\begin{theorem}
\label{FLtheo}
A strict GFN model $\mathcal{Q}$ is stable if and only if it admits a Lyapunov function. In particular, $V$ can be chosen as
\begin{equation*}
V(x) = \sup_{ Q (\cdot) \in \mathcal{Q}_x}  \int_0^{\infty} \| Q(s)\| ds
\end{equation*}
and $V$ is continuous.
\end{theorem}

\begin{proof}
First we show that the existence of a Lyapunov function is sufficient for stability. Let $V$ be a Lyapunov function for $\mathcal{Q}$. From (\ref{fLF1}) it follows that $V(Q(t)) \geq 0$ and inequality (\ref{fLF2}) implies that
\begin{equation*}
 \qquad V(Q(t_2)) - V(Q(t_1)) \leq 0
\end{equation*}
for all $t_1 \leq t_2 \in \R_+$. So $V(Q(\cdot))$ is monotone decreasing and bounded. In order to show that $V(Q(t))$ tends to zero as $t$ goes to infinity assume that
\begin{equation*}
\lim_{t \rightarrow \infty} V(Q(t)) = : c >0.
\end{equation*}
Then for all $t\geq 0$ it holds that
\begin{equation}\label{+}
0 < c \leq V(Q(t)) \leq w_2(\|Q(t)\|)
\end{equation}
and further $0 < w_2^{-1}(c) \leq \|Q(t)\|$. It also holds that
 $$0 < w_3(w_2^{-1}(c)) \leq w_3(\|Q(t)\|).$$ 
Now observe that from (\ref{fLF2}) it follows that
\begin{align*}
V(Q(t)) - V(Q(0)) &\leq -\,\int_0^t w_3(\|Q(s)\|) ds  \leq -\int_0^t   w_3(w_2^{-1}(c)) ds  \leq -\, w_3(w_2^{-1}(c))\,t 
\end{align*}
and hence $ \lim_{t \rightarrow \infty} V(Q(t)) = - \infty$, which is a contradiction to (\ref{+}). Consequently
\begin{align}\label{firststep}
\lim_{t \rightarrow \infty} V(Q(t)) = 0.
\end{align}
By (\ref{fLF1}) it follows that
\begin{align*}
 \lim_{t \rightarrow \infty} \|Q(t)\| = 0.
\end{align*} 
So the zero path is asymptotically stable and this implies by Lemma~\ref{stabdefequiv} the stability of the strict GFN model $\Q$.


Conversely suppose that $\mathcal{Q}$ is stable. Then there is a $\tau >0$ such that $Q(\tau + \cdot) \equiv 0$ for all paths $Q(\cdot) \in \mathcal{Q}(1)$. We define the following comparison functions
\begin{align*}
w_1(r) := \frac{r^2}{2L}, \quad  w_2(r) := r^2\,(1 + L\tau)\, \tau ,\quad w_3(r) := r
\end{align*}
and show that our candidate 
\begin{equation*}
V(x) = \sup_{ Q(\cdot) \in \mathcal{Q}_x}  \int_0^{\infty} \| Q(s)\| ds
\end{equation*}
is a Lyapunov function. As $\mathcal{Q}$ satisfies the Lipschitz condition (a) it follows that
\begin{align} \label{lip2}
 \| Q(s)\| \geq \| Q(t)\| \,- L( s-t)
\end{align}
for all $Q \in \Q$ and  $s \geq t$. In particular for $t=0$ 
this implies 
\begin{align} \label{lip2b}
 \|Q(s)\| \geq \| Q(0)\| \,- L s .
\end{align}
Using the last inequality we get the following estimate from below
\begin{align*}
V(x) &= \sup_{ Q(\cdot) \in \mathcal{Q}_{x}} \int_0^{\infty} \|Q(s)\| \,ds \, 
\geq \sup_{ Q (\cdot) \in \mathcal{Q}_{x}} \int_0^{\frac{\|x\|}{L}} \| Q(s)\| \,ds \, \,\\ &\geq \, \,
 \sup_{ Q(\cdot) \in \mathcal{Q}_{x}} \int_0^{\frac{\|x\|}{L}} \left(\|x\|\,-\,Ls \,\right)\,\,ds \\&
=\sup_{ Q (\cdot)\in \mathcal{Q}_{x}} \left\{ \|x\| \frac{\|x\|}{L} - \frac{\|x\|^2}{2L}  \right\}  
=   \frac{\|x\|^2}{2L} = w_1(\|x\|).
\end{align*}
To obtain an estimate from above consider $ Q(\cdot) \in \mathcal{Q}_{x}$. Note that by the scaling property it follows that
$
\frac{1}{\|x\|}\,  Q(\|x\| \cdot) \in \mathcal{Q}(1)
$ 
and further the stability of $\mathcal{Q}$ implies that
\begin{equation}\label{ss}
 Q(s) =0 \qquad \forall \, s \geq \|x\|\tau.
\end{equation}
The triangle inequality together with the Lipschitz condition imply that for all $s \in [0,\|x\|\tau]$ it holds that
\begin{align}\label{lip1}
\| Q(s)\| \leq \| Q(0)\| \, + L\|x\|\tau = \|x\| \,(1 + L\tau).
\end{align}
With (\ref{ss}) and (\ref{lip1}) an estimate from above is derived as follows
\begin{align*}
V(x) = \sup_{ Q(\cdot) \in \mathcal{Q}_{x}} \int_0^{\|x\|\tau} \|Q(s)\| \,ds 
&\leq \, \sup_{ Q(\cdot) \in \mathcal{Q}_{x}} \int_0^{\|x\|\tau} \|x\|\,(1 + L\, \tau) \,ds  \\
&= \|x\|^2\,(1 + L\tau)\, \tau 
=  w_2(\|x\|).
\end{align*}
Now consider the decrease condition
\begin{align*}
V(Q(t_2)) - V(Q(t_1)) = 
\sup_{ Q(\cdot) \in \mathcal{Q}_{Q(t_2)}} \int_0^{\infty} \| Q(s)\| \,ds -\sup_{ Q (\cdot)\in \mathcal{Q}_{Q(t_1)}} \int_0^{\infty} \| Q(s)\| \,ds.
\end{align*}
From condition (e) it follows that
\begin{align*}
V(Q(t_1)) &= \sup_{Q (\cdot)\in \mathcal{Q}_{Q(t_1)}} \int_0^{\infty} \| Q(s)\| \,ds
 \\&\geq   \int_{t_1}^{t_2} \| Q(s)\| \,ds + \sup_{ Q (\cdot)\in \mathcal{Q}_{Q(t_2)}} \int_0^{\infty} \| Q(s)\|\,ds\\
&=  \int_{t_1}^{t_2} \| Q(s)\| \,ds + V(Q(t_2)).
\end{align*}
and hence
\begin{align*}
V(Q(t_2)) - V(Q(t_1)) \leq -\,\int_{t_1}^{t_2} \|Q(s)\|\,ds =  -\,\int_{t_1}^{t_2} w_3(\|Q(s)\|)\,ds.
\end{align*}
Thus together with Proposition~\ref{Vcont} we see that $V$ is a Lyapunov function.

\end{proof}

From the proof of the previous theorem we see that the semicontinuity
property (f) is only needed to conclude continuity of $V$. Thus we have
also proved

\begin{corollary}
A closed GFN model $\Phi$ that satisfies the concatenation property (e) is stable if and only if it admits a Lyapunov function. In particular $V$ can be chosen as in \eqref{lyapfunc} and $V$ is upper semicontinuous.
\end{corollary}

\begin{remark}
Since upper semicontinuous Lyapunov function do not 
imply robustness statements the benefit of Lyapunov functions that are upper semicontinuous is restricted compared to continuous Lyapunov functions, see \cite{kelletteel04}, \cite{teel2000smooth}.
\end{remark}

\section{Fluid networks as differential inclusions}
\label{di-elements}

We want to apply the main theorem to fluid network models that work under a specific discipline. So we 
need to show that the additional conditions (e) and (f) are satisfied in each case. In order to 
obtain condition (e) we make use of concepts 
from the theory of differential inclusions. Clearly a detailed description of the dynamics of a fluid network depends on the specific discipline that is used. But one part of the dynamics of fluid network models that 
all service disciplines have 
in common 
is the so called flow balance relation
\begin{align}\label{be}
 Q (t) = Q(0) +  \alpha t - (I-P^T)\, M  T(t).
\end{align}
Here $\alpha \in \R_+^K$ represents the inflow rate, $\mu \in \R_+^K$ denotes the outflow rate, $M=\text{diag}(\mu)$ and $P$ is the routing matrix. The initial value or level of the fluid network is given by $Q(0)=x$. A basic property of the fluid level process $Q$ as well as the allocation process $T(t)$ is that both processes are Lipschitz continuous \cite{chen} and hence differentiable almost everywhere. So for almost all $t\in \mathbb{R}_+$ the flow balance relation \eqref{be} can also be written as
\begin{align}\label{diffbalance}
 \dot Q (t) =  \alpha - (I-P^T)\, M \dot T(t), \qquad Q(0)=x.
\end{align}
Now we consider the derivative of the allocation process as the control
variable, i.e. we define $u(t):=\dot T(t)$ a.e. . Note that $u$ is
measurable. The allocation process is determined through the service
discipline. So each service discipline has a set of admissible controls
$U(Q)$, where $u \in U(Q)$ if and only if $u \in \R_+^K$ satisfies some
allocation conditions that are specific to the discipline. As mentioned in
\cite{chen} the allocation process need not be unique and so for every
$Q\in \mathbb{R}_+^K$ there are different choices of $u$ possible. But
the admissible control values $u$ depend on the fluid level process $Q(t)$
through the allocation conditions. Consequently we consider the set of
admissible control values as a set $U(Q).$ Thus, the flow balance
relation \eqref{diffbalance} can also be expressed by a differential
inclusion of the form
\begin{align*}
 \dot Q (t) = \alpha - (I-P^T)\, M u(t) =: f(Q(t),u(t)), \quad Q(0)=Q_0
\end{align*}
with $u(t) \in U(Q(t))$. Often $U$ is referred to as the feedback map. By setting
\begin{align}\label{diffinclGdef}
 F(Q) = \{f(Q,u) \, : \, u \in U(Q)\} 
\end{align}
we rewrite this as a closed loop differential inclusion
\begin{align}\label{diffinclGcloop}
\dot Q(t) \in F(Q(t)), \quad Q(0)=Q_0.
\end{align}

In the following we state some results from the theory of differential inclusions that will be useful to show that specific fluid networks satisfy the conditions (e) and (f). 
Let $K\subset \R^n$ and consider the differential inclusion 
\begin{align}\label{diffinclF}
\dot x(t) \in F(x(t)).
\end{align}
Let $\mathcal{S}_F(x_0)$ denote the set of solutions to \eqref{diffinclF} starting at $x_0 \in K$. The existence theorem is as follows \cite[Theorem 5.2]{smirnov}.
\begin{theorem}\label{exists}
Let $K \subset \R^n$ be a closed set. Assume that the set-valued map $F: K \rightsquigarrow \R^n$ with closed convex values contained in a ball of radius $b>0$ is upper semicontinuous. Then the following conditions are equivalent.
\begin{enumerate}
\item[(1)] For any $x_0 \in K$ there is a solution $x(\cdot) \in \mathcal{S}_F(x_0)$ satisfying $x(t) \in K$ for all $t\geq 0$.
\item[(2)] For any $x \in K$ it holds that $F(x) \cap \mathcal{T}_K(x)\not = \emptyset$.
\end{enumerate}
\end{theorem}
Here $\mathcal{T}_K(x)$ denotes the contingent cone to $K\subset \R^n$ at $x$, which is defined as the set of $v  \in \R^n$ such that there is a sequence $(h_n)_{n\in \N} \subset \mbox{int}(\R_+)$ converging to $0$ and a sequence $(v_n)_{n \in \N} \subset \R^n$ converging to $v$ such that for all $n \in \N$ it holds that $x+h_n\,v_n \in K$. 
A useful criterion to conclude upper semicontinuity of a parameterized set-valued map is the following \cite[Proposition~1.4.14]{aubin-frankowska}.
\begin{proposition}\label{prop:para-set-valued-map}
Let $X,Y$ and $Z$ be metric spaces and $U:X\rightsquigarrow Z$ be a set-valued map. Assume that $f:\mbox{graph}(U) \rightarrow Y$ is continuous. If $U$ is upper semicontinuous with compact values then $F:X \rightsquigarrow Y$ defined by
\begin{align*}
F(x) :=\{ f(x,u) \, : \, u\in U(x)\}.
\end{align*}
is upper semicontinuous.
\end{proposition}

\section{Applications to some fluid networks} \label{appl-fn}

In this section we show that our main result can be applied to some special fluid networks. In particular, we show that fluid networks under general work-conserving and priority disciplines satisfy the additional conditions (e) and (f) given in Definition~\ref{mGFN}.
The following description of a fluid network is taken from \cite{Yechen}. A fluid network consists of $K$ different fluid classes and $J$ stations, where the fluids are served. There is a (not necessarily injective) map $s$ that prescribes which fluid class is served at which station. Fluid class $k$ is exclusively served at station $s(k)$. For every station the set $C(j) := \{ k \in \{1,...,K\} \, : \, s(k) =j \}$ can without loss of generality assumed to be nonempty. The corresponding $J\times K$ matrix $C$ is called the constituency matrix, where $c_{jk} =1$ if $s(k)=j$ and zero else. Further we introduce two nonnegative vectors $\alpha,\mu \in \R_+^K$ and a $K \times K$ substochastic matrix $P$. Where $\alpha_k$ denotes the exogenous inflow rate of fluid class $k$ and $\mu_k$ denotes the potential outflow rate of fluid class $k$. The matrix $P$ will be referred to as the routing matrix. The element $p_{kl}$ of $P$ denotes the proportion of the outflow of class $k$ which turns into fluid class $l$. So $1- \sum_{l=1}^K p_{kl}$ is the part of the outflow of class $k$ that leaves the network. The routing matrix is assumed to have spectral radius strictly less than one, i.e. all fluids eventually leave the network. The initial fluid level is represented through the $K$-dimensional vector $Q_0$. The fluid network is described by $(\alpha,\mu,P,C)$ with initial fluid level $Q_0$. The time-evolution is described by the $K$-dimensional fluid level process $\{Q(t) \, : \, t\geq 0\}$ and the $K$-dimensional allocation process $\{T(t)\, : \, t \geq 0\}$, where $Q_k(t)$ denotes the amount of class $k$ fluids in the network at time $t$ and $T_k(t)$ denotes the total amount of time during the interval $[0,t]$ that station $s(k)$ has devoted to serve fluid class $k$. We note that the processes are Lipschitz continuous and hence differentiable almost everywhere by Rademacher's Theorem. A precise description of the dynamics of a fluid network depends on the service discipline.

\subsection{Fluid networks under general work-conserving disciplines}

The dynamics of a fluid network under a general work-conserving service discipline can be summarized as follows
\begin{align}\label{D1}
 Q(t) &= Q_0 + \alpha\,t -(I-P^T) M T(t) \geq 0,\\ \label{D2}
{T}(0) &=0 \text{  and  } {T}(\cdot) \text{ is nondecreasing, } \\\label{D3}
{I}(t) &= et -C\,{T}(t) \text{  and  } {I}( \cdot)\text{ is nondecreasing, }\\\label{D4}
0&=\int_0^{\infty} (C\, {Q}(t) )^T \,\,d{I}(t),
\end{align}
where $M = \text{diag}(\mu)$. Equation~(\ref{D4}) describes the work-conserving property of the network and relation (\ref{D1}) is called the flow balance relation. In general the allocation process is not unique. Any pair $({Q}(\cdot),{T}(\cdot))$ that satisfies (\ref{D1})-(\ref{D4}) is called fluid solution of the work-conserving fluid network. The set of all feasible fluid level processes is denoted as
\begin{align*}
\mathcal{Q}_C=\{ Q(\cdot) : \, \exists \, T(\cdot) \mbox{ such that } (Q(\cdot),T(\cdot)) \text{ is a fluid solution}\, \}.
\end{align*}
To prove the existence of a work-conserving allocation process we bring the conditions \eqref{D2}-\eqref{D4} into the context of differential inclusions. To this end, we define $\dot T(t)=:u(t)$ and consider the differential form of the flow balance equation 
\begin{align}\label{diffbalance-wc}
 \dot Q (t) =  \alpha - (I-P^T)\, M u(t).
\end{align}
For $Q\in \R^K_+$ the conditions 
defining the admissible values of $u$ are
\begin{align}\label{U1}
 u \geq 0,\qquad
e-Cu\geq0,\qquad
(CQ)^T\cdot(e-Cu) = 0\,.
\end{align}
These are immediate consequences of \eqref{D2}, \eqref{D3} and \eqref{D4}
in their differentiation. Note the discontinuity of these conditions on
the boundary of $\R^K_+$, because in this case zeros may appear in
$(CQ)^T$. Now the set of admissible controls is
\[
 U_C(Q) := \left\{ u \in \mathbb{R}^K \, :\,  \eqref{U1} \text{ is satisfied }\right\}.
\]
Using $f(Q,u) := \alpha - (I-P^T)Mu$ this leads to a differential inclusion of the form 
\begin{align}\label{diff-incl-Wc}
\dot Q(t) \in  \left\{ f(Q(t),u) \, :\ u \in U_C(Q(t))\right\}.
\end{align} 
For brevity, we define the following set-valued map $F: \R_+^K \rightsquigarrow \R^K$ by
\begin{align}\label{setvaluedG}
F(Q) = \left\{ f(Q,u) \, :\ u \in U_C(Q)\right\}
\end{align}
so that the corresponding differential inclusion compactly reads as
\begin{align}\label{diffinclG}
\dot Q(t) \in F(Q(t)), \quad Q(0)=Q_0.
\end{align}
Using this approach we are able to give an 
alternative proof for the Theorem 2.1 in \cite{chen}.

\begin{theorem}\label{existence}
For any work-conserving fluid network $(\alpha,\mu, P, C)$ with an initial level $Q_0$ the set $\mathcal{Q}_C$ is nonempty.
\end{theorem}

\begin{proof}
From the conditions \eqref{U1} it follows that the set $U_C(Q)$ is compact and convex and upper semicontinuous. Further, the set-valued map $U_C(\cdot)$ is upper semicontinuous and $f(Q,u)$ is continuous. Hence, by Proposition~\ref{prop:para-set-valued-map} the set-valued map $F$ is upper semicontinuous. Moreover, $F$ has closed convex values that are contained in some ball with radius $b>0$. Also the conditions \eqref{U1} imply that $F(Q) \cap \mathcal{T}_{\R_+^K}(Q) \not = \emptyset$ for all $Q\in \R_+^K$. Then by Theorem~\ref{exists} there exists a solution to \eqref{diffinclG}.
To show the existence of an allocation $T$ let $Q(\cdot)$ be a solution to \eqref{diffinclG}. Note that $f(Q,u)$ is continuous in $u$ and the that $U(t):=\{ u \in \R_+^K\, : \, e-Cu\geq 0,\, \, (C\, Q(t))^T \, (e-Cu)=0\}$ is closed and bounded. Also, we note that $t \rightsquigarrow U(t)$ is upper semicontinuous. Then, by the Filippov measurable selection Lemma in
\cite[p. 78/79]{filippov}, there is a measurable selection $u(\cdot)$ of
$u(t) \in U_C(Q(t))$ such that
\begin{align*}
\dot Q (t) = \alpha - (I-P^T)M\,u(t) \qquad \mbox{ for almost all } \, t \geq 0.
\end{align*}
Thus, integrating the latter yields that, given the initial value $Q_0$, the pair $(Q(\cdot),T(\cdot))$ with $T(t) := \int_0^t u(s)ds$ is a fluid solution. 

\end{proof}
So, we can represent the set of work-conserving fluid level processes by
\[
 \mathcal{Q}_C = \{ Q(\cdot) \in \mathcal{S}_{F}(Q_0) \, :\, Q_0 \in \R_+^K \, \}.
\]
In \cite{Yechen} it is shown that $\mathcal{Q}_C$ defines a closed GFN model. So we only have to prove that the conditions~(e) and (f) are satisfied.

\begin{proposition}\label{comprop}
The set of fluid level processes $\mathcal{Q}_C$ satisfies the concatenation property.
\end{proposition}

\begin{proof}
    Since solutions of differential inclusions are by definition
    absolutely continuous functions, and concatenation preserves absolute
    continuity the assertion follows.

\end{proof}

To show that condition~(f) is satisfied we need to show that the solution map is lower semicontinuous.

\begin{theorem}\label{perfectness}
The set of fluid level processes $\mathcal{Q}_C$ satisfies the lower semicontinuity property~(f).
\end{theorem}

\begin{proof}
    To show condition~(f) we have to verify the existence of a $T>0$ such
    that $Q_0 \rightsquigarrow \mathcal{S}_F (Q_0) {|}_{[0,T]}$ is lower
    semicontinuous. 
    In view of Remark~\ref{semicontremark} and Proposition~\ref{comprop} it is sufficient to construct
    for each $Q_0$ a $T(Q_0)>0$ such that (f') holds.

To this end, let $Q_0 \in \mathbb{R}_+^K$ be fixed, $Q(\cdot) \in \mathcal{S}_F(Q_0)$. Then, by the proof of Theorem~\ref{existence} there exists a function $u(\cdot) \in U(Q(\cdot))$ such that
\begin{align}\label{eq:lsc-proof-1}
 Q (t) = Q_0 + \alpha\, t - (I-P^T)M\,\int_0^t u(s)ds.
\end{align}
We 
%
distinguish the following situations.

First, suppose that $Q_0 \in \R_+^K$ and all stations have some nonempty queues, i.e. $CQ_0 >0$. Hence, there is a $T(Q_0)>0$ such that $CQ(t)>0$ for all $t \in [0,T(Q_0)]$. We note that $(CQ)^T\cdot(e-Cu) = 0$ from \eqref{U1} also reads as
\begin{align*}
\sum_{j =1}^J  \left( \sum_{l \in C(j)} Q_l  \cdot  \left(1 -\sum_{l \in C(j)}  u_l\right) \right) =0.
\end{align*}
Since both factors are nonnegative and, in fact, $CQ(t)>0$ for $t \in [0,T(Q_0)]$, it holds that 
\begin{align}\label{s1}
1 =\sum_{l \in C(j)} u_l(\cdot)|_{[0,T(Q_0)]}=:e_j^T C u(\cdot)|_{[0,T(Q_0)]}
\end{align}
for all $j=1,...,J$. 
Let $(Q^n_0)_{n \in \mathbb{N}}$ be a sequence of initial values which converges to $Q_0$ and define $ \delta:= \min_{k=1,...,K} Q_{0_k} $. Furthermore, let $0<\delta'<\delta$ and $N(\delta') \in \mathbb{N}$ be such that for all $n \geq N(\delta')$ we have that
$\| Q_0^n - Q_0\| < \delta'$. In particular, it holds that $Q^n_0 \in \mbox{int}(\mathbb{R}_+^K)$ for all $n \geq N(\delta')$. Hence, there is $T>0$ such that $Q(\cdot)|_{[0,T]} \in \mbox{int}(\R_+^K)$ and $Q^n(\cdot)|_{[0,T]} \in \mbox{int}(\R_+^K)$ for all $n \geq N(\delta')$. Now consider the paths
\begin{align*}
 Q^n (t) := Q^n_0 + \alpha\, t - (I-P^T)M\,\int_0^t u^n(s)ds,
\end{align*}
where for $n \geq N(\delta')$ we define $u^n(s):=u(s)$. For $n$ sufficiently large we have that $Q^n(t)\in \mbox{int}(\mathbb{R}_+^K)$ for $t \in [0, T(Q_0)]$, because
\begin{align}\label{eq-uni-con}
\lim_{n \rightarrow \infty} \sup_{t \in [0,T(Q_0)]} \| Q^n (t) - Q(t)\| =0.
\end{align}
Consequently, $Q \rightsquigarrow \mathcal{S}_F (Q) {|}_{[0,T(Q_0)]}$ is lower semicontinuous at $Q_0$.

Second, suppose that the initial fluid level at some stations is zero.
We first treat the case of a single station with empty queues. Without loss of generality let this station be $j=1$ and let $a$ denote the set of classes which are served at station $1$. Then, the last constraint in \eqref{U1} is not active for station $1$ and therefore the constraints for fluid classes $k \in a$ are given by
\begin{align}\label{cond-u-proof-lsc}
 u_k  \geq 0 ,\qquad 1 - \sum_{l \in a} u_l \geq 0.
\end{align}
However, since $Q(\cdot)$ is a solution to the differential inclusion~\eqref{diffinclG} potentially only a proper subset of \eqref{cond-u-proof-lsc} is feasible. 
If this condition enforces equality in the second constraint in \eqref{cond-u-proof-lsc}, then we can argue as in \eqref{s1} on a sufficiently small time interval and the proof of \eqref{eq-uni-con} applies again. The interesting case is when there is idle capacity at station $j=1$. Here  $u_k(\cdot)\geq 0$ are such that $\sum_{l \in a} u_l (\cdot) <1 $ and that the fluid levels of classes $k \in a$ remain nonnegative. Using $b := \{1,...,K\} \backslash a$ the differential form of the flow balance equation \eqref{D1} can be expressed in block form by
\begin{align*}
\begin{bmatrix}
\dot Q_a(t)\\
\dot Q_b(t) 
\end{bmatrix}
=
\begin{bmatrix}
\alpha_a\\
\alpha_b
\end{bmatrix}
+
\begin{bmatrix}
 P^T_a & P^T_{ab}\\
P^T_{ba} & P^T_b
\end{bmatrix}
\,\begin{bmatrix}
 M_a & 0\\
0& M_b
\end{bmatrix}
\,
\begin{bmatrix}
 u_a(t)\\
u_b(t)
\end{bmatrix}
- 
\begin{bmatrix}
 M_a & 0\\
0& M_b
\end{bmatrix}
\,
\begin{bmatrix}
 u_a(t)\\
u_b (t)
\end{bmatrix}.
\end{align*}
The nonnegativity of the fluid levels for classes $l \in a$ yields the following condition
\begin{align*}
0 \leq \alpha_a + P_{ab}^T \, M_b\, u_b(\cdot) - (I_a - P_a^T) \, M_a\, u_a(\cdot), 
\end{align*}
which also reads as 
\begin{align*}
u_a(\cdot) \leq  M_a^{-1}\, (I_a - P_a^T)^{-1}\, (\alpha_a + P_{ab}^T \, M_b\, u_b(\cdot)).
\end{align*}
As $e_j^TCQ_0 >0$ for $j \not = 1$ then, arguing as in \eqref{s1} there is a $T(Q_0)>0$ such that the allocation rates corresponding to fluid classes present at the stations $j \not =1$ satisfy $\sum_{l\in C(j)} u_l(\cdot)|_{[0,T(Q_0)]} =1$. 
Let $\varepsilon >0$ be fixed, so that if $\|Q_0 - Q\| < \varepsilon$ then $Q_k>0$ when $Q_{0,k}>0$.
Now, for another initial value  $Q^1_0$ with $\|Q_0 - Q_0^1\|< \varepsilon$ we consider $u^1(\cdot):=(u_a(\cdot) + v(\cdot)\, \, u_b(\cdot))^T$, where $v(\cdot)$ takes values in $\mathbb{R}^{|a|}$ such that
\begin{align}\label{cond-v-n}
\sum_{l \in a} u_l (t)+ {v}_l (t) =1 \quad \mbox{ if } \quad e_1^T C Q^1(t)>0,
\end{align}
and $v(t)=0$ otherwise.
Then, we consider the solution $Q^1(\cdot)$ associated with $u^1(\cdot)$ and $Q^1_0$, i.e. 
\begin{align*}
 Q^1 (t) &= Q_1 + \alpha\, t - (I-P^T)M\,\int_0^t \begin{bmatrix}
 u_a(s)+ v(s)\\
u_b(s)
\end{bmatrix} ds\\
         &= Q_1 -Q_0 + Q(t) - (I-P^T)M\,\int_0^t \begin{bmatrix}
 v(s)\\
0
\end{bmatrix}ds.
\end{align*}
So, the difference between the solutions $Q(\cdot)$ and $Q^1(\cdot)$ is given by
\begin{align*}
 Q^1 (t) - Q(t) = Q^1_0 -Q_0  +\int_0^t \begin{bmatrix}
P_a^T M_a v(s)\\
P_{ba}^T M_a v(s) 
\end{bmatrix} -  \begin{bmatrix}
M_a v(s)\\
0 
\end{bmatrix}ds
\end{align*}
In particular, as $Q_{0,a}=0$ we have that
\begin{align*}
 Q_a^1 (t) - Q_a(t) = Q^1_{0,a} - (I-P_a^T)M_a\,\int_0^t v(s) \, ds.
\end{align*}
Hence, if $Q_{0,a}^1>0$ the nonnegativity of $(I-P_a^T)M_a$ and $v(\cdot)$ imply that there is a $r\geq 0$ such that 
\begin{align}\label{eq:null}
Q^1_{0,a} - (I-P_a^T)M_a\,\int_0^r v(s) \, ds = 0.
\end{align}
We will assume that $v(\cdot)$ is chosen so that the time in which \eqref{eq:null} is achieved is minimal.

Thus, given  a sequence of initial values $(Q^n_0)_{n \in \mathbb{N}}$ converging to $Q_0$ and in particular $Q^n_{0,a}$ converging to zero, we define 
\[
r_n:=\min\{ r\geq 0 \, : \,v^n(\cdot) \mbox{ satisfies }  \eqref{cond-v-n} \mbox{ and } \eqref{eq:null}\}
\]
and
\begin{equation*}
u^n(t):= \begin{cases}
(\, u_a(t) + v^n(t)\, \, u_b(t)\, )^T & \quad  \text{for } \quad 0\leq t \leq r_n, \\
(\, u_a(t)\, \, u_b(t)\, )^T & \quad  \text{for } \quad t > r_n.
\end{cases}
\end{equation*}
Further, we note that \eqref{eq:null} implies that if $Q^n_{0,a}$ converges to $Q_{0,a}=0$ it holds that $r_n$ converges to zero as well. Hence, we have that $u^n(\cdot)$ converges to $u(\cdot)$ and consequently $Q^n(\cdot)$ converges uniformly to $Q(\cdot)$ on $[0,T(Q_0)]$, i.e. $Q(\cdot)|_{[0,T(Q_0)]}$ depends lower semicontinuously on $Q_0$.

The cases where more than one stations have empty queues follows the same line of reasoning.
Finally, the assertion follows from Remark~\ref{semicontremark}.

\end{proof}

Summarizing we obtain.

\begin{theorem}
General work-conserving fluid networks define strict GFN
models. In particular, it is stable if and only if it admits a continuous Lyapunov function.
\end{theorem}

\subsection{Fluid networks under priority disciplines}

The priority service discipline 
%
assigns different priorities to the fluid classes that are served at one station,
\cite{Yechen}. This is done via a permutation mapping $ \pi : \{1,...,K\}
\rightarrow \{1,...,K\}.  $ To be precise, let $s(l)=s(k)$ for $l,k \in
\{1,...,K\}$ then fluids of class $l$ have higher priority than fluids of
class $k$, if $\pi(l)<\pi(k)$.  That is, fluids of class $k$ are not
served as long as the fluid level of class $l$ is greater than
zero. For each $k \in \{1,...,K\}$ the set of fluid classes that are
served at the same location $s(k)$ and have higher priority is denoted
by $ \Pi_k := \{ l\, \, \in \{1,...,K\}: l \in C(s(k)), \, \pi(l) \leq
\pi(k) \}.  $ To derive a description of fluid networks under the priority
discipline $\pi$ we introduce the unused capacity process $Y(t)$. Namely,
$Y_k(t)$ is denotes the cumulative remaining capacity of location $s(k)$
for serving fluids of classes that have strictly lower priority than
fluids of class $k$. The dynamics can be described as follows
\begin{align}\label{Dp1}
 Q(t) &= Q_0 + \alpha\,t -(I-P^T) M T(t) \geq 0,\\ \label{Dp2}
{T}(0) &=0 \text{  and  } {T}(\cdot) \text{ is nondecreasing, } \\\label{Dp3}
Y_k(t) &= t-\sum_{l \in \Pi_k} T_l(t) \text{  and  } {Y}( \cdot) \text{ is nondecreasing, }k \in \{1,...,K\}\\\label{Dp4}
0&=\int_0^{\infty} Q_k(t) \,\,dY_k(t), \qquad k\in \{1,...,K\}.
\end{align}
Any pair $({Q}(\cdot),{T}(\cdot))$ that satisfies (\ref{Dp1})-(\ref{Dp4}) is called a fluid solution of the fluid network under the priority discipline $\pi$. The set of all feasible fluid level processes is denoted as
\begin{align*}
\mathcal{Q}_P=\{ Q(\cdot) : \, \exists \, T(\cdot) \mbox{ such that } (Q(\cdot),T(\cdot)) \text{ is a fluid solution}\, \}.
\end{align*}
Again we bring this into the context of differential inclusions by setting $\dot T(t)=u(t)$. The constraints for $k \in \{1,...,K\}$ are here
\begin{align}\label{Up1}
 u_k \geq 0,\quad 
1-\sum_{l \in \Pi_k} u_l\geq0, \quad 
Q_k \cdot (1-\sum_{l \in \Pi_k} u_l) = 0
\end{align}
and the set of admissible controls is
\begin{align*}
 U_P(Q) := \left\{ u \in \mathbb{R}_+^K \, :\,  \eqref{Up1} \text{ is
     satisfied for all } k \in \{1,...,K\} \right\}.
\end{align*}

Following the same line of 
reasoning we conclude the following.

\begin{theorem}\label{existenceandperfectness}
The set $\mathcal{Q}_P$ is nonempty and satisfies the lower semicontinuity property.
\end{theorem}

In order to prove that $\mathcal{Q}_P$ is a strict GFN model it remains to show that the concatenation property holds, as the validity of the conditions (a)-(d) is shown in \cite[Lemma 3.5]{Yechen}. Using results from differential inclusions we obtain the following result.

\begin{proposition}\label{comproppriority}
The set $\mathcal{Q}_P$ satisfies the concatenation property. 
\end{proposition}

Thus we may conclude.

\begin{theorem}
The fluid network under priority discipline $\mathcal{Q}_P$ is a strict GFN model. It is stable if and only if it admits a continuous Lyapunov function.
\end{theorem}

\section{Fluid limit models of queueing networks}
\label{sec:fluidlimit}

A further class of interest are fluid limit models of queueing networks. For
this class the open question remained whether they define closed GFN models
\cite{Yechen}. As we will see, taking the closure with respect to uniform
convergence on compact sets does not change the stability properties.  In this
way we obtain from fluid limit models closed GFN models.
We state a condition for which we conjecture that it guarantees condition
(e) but so far a proof has remained elusive.

A queueing network consists of $J$ stations that serve $K$ classes of
customers. For each class $k \in \{1,...,K\}$ the interarrival times are
denoted by $\{\xi_k(n) : n \geq1\}$ and the service times are given by
$\{\eta_k(n) : n \geq1\}$, 
where $n \in \mathbb{N}$ denotes the place in
the sequence of customers of the considered class. It is possible that for
some customer classes $k$ no exogenous arrivals take place, then the
interarrival time $\xi_k(n)=\infty$ for all $n$. The
set of customers with exogenous arrivals is denoted by
\begin{equation*}
\mathcal{E}:= \{k \in \{1,...,K\} : \xi_k(n) < \infty, \, n \geq1\}.
\end{equation*}
Further the waiting buffer at each station is assumed to have infinite capacity. The random variables above are defined on some probability space $(\Omega, \mathcal{F},\p)$. The following assumptions on the interarrival times $\xi_k$ and service times $\eta_k$ are made.
\begin{align}\label{A1}\tag{A1}
\xi_1,...,\xi_K, \eta_1,...,\eta_K \text{ are i.i.d. and mutually independent.} 
\end{align}
The first moments are assumed to be finite, i.e.

\begin{equation}\label{A2}\tag{A2}
\begin{split}
\alpha_k^{-1} &= \E [\,\xi_k(1)\,] < \infty \quad \forall \,k \in \{1,...,K\}, \\  \mu_k^{-1}&=\E [\,\eta_k(1)\,] < \infty \quad \forall \,k \in \mathcal{E}.
\end{split}
\end{equation}

The interarrival times are 
assumed to be unbounded and spread out, i.e. for each $k \in \mathcal{E}$
there exists some integer $j_k \in \N$ and some function $p_k:\R_+
\rightarrow \R_+$ with $\int_0^{\infty} p_k(x) dx >0$, such that $
\p[\,\xi_k(1) \geq x\,] >0$ for all $x >0$ and
\begin{align}\label{A3}\tag{A3}
\p \left[ a \leq \sum_{i=1}^{j_k} \xi_k(i) \leq b \right]\,\, \geq \, \,\int_a^b p_k(x)dx \quad \forall \, 0 \leq a <b.
\end{align}
Let $\phi^k(n)$ be the routing vector for the $n$th customer of class $k$ who finishes service at the station $s(k)$. So $\phi^k(n)$ is a $K$-dimensional Bernoulli random variable with parameter $P_k^T$. The corresponding routing matrix $P$ is assumed to have spectral radius strictly less that one. Further, it is assumed that for each $k \in \{1,...,K\}$ the routing process
\begin{align*}
 \phi^k = \{\phi^k(n):n\geq 1\}
\end{align*}
is i.i.d., $\phi^1,...,\phi^K$ are independent and independent of the arrival processes and service processes. 
The evolution of the queueing network is described by a Markov process $X=\{X(t), t\geq 0\}$ on a measurable state space $(\X,\B_{\X})$ that is defined on a measurable space $(\Omega,\F)$. Further, $X$ is adapted to the filtration $(\F_t)_{t\in \R_+}$ and the probability measures $\{\p_x,x\in \X \}$ on $(\Omega, \F)$ satisfy $\p_x[X(0)=x]=1$ for all $x\in \X$. 
In general the states are given by points 
\[\X \subset \Z_+^{\infty} \times \R_+^{2K + |\mathcal{E}|},\]
where $|\mathcal{E}|$ denotes the cardinality of $\mathcal{E}$ and $\Z_+^{\infty}$ denotes the set of finitely terminating sequences in $\Z_K=\{1,2,...,K\}$. For instance, for priority queueing networks the state space is a subset of $\Z_+^K \times \R_+^{K + |\mathcal{E}|}$. For further details see \cite{bramsonLN,dai}.
Analogous to the fluid models there is a set of equations that embraces most of the network dynamics. Consequently, ${Q^x}(t) \in D(\R_+,\R_+^K)$ denotes the queue length process and ${T}^x(t) \in D(\R_+,\R_+^K)$ denotes the allocation process. The superscript $x$ expresses the initial state $x=(q,u,v)$, where vectors $q,u$ and $v$ denote the queue length, the residual interarrival time, and the residual service time.
%
%
%
Consider a pair of sequence $(r_n,x_n)_{n \in \N}$, where $x_n \in \X$ is a sequence of initial states and $r_n \in \mathbb{R}_+$ such that
\begin{align}\label{sequences}
\lim_{n \rightarrow \infty} r_n= \infty \, , \quad 
\quad \limsup_{n \rightarrow \infty} \frac{\|q_n\|}{r_n} < \infty\, , \quad 
\quad \lim_{n \rightarrow \infty} \frac{\|u_n\|}{r_n} =  \lim_{n \rightarrow \infty} \frac{\|v_n\|}{r_n} = 0.
\end{align}
In \cite{bramsonLN,dai} it is shown that under the assumptions \eqref{A1}-\eqref{A3} for almost all sample paths $\omega \in \Omega$ and any pair of sequence $(r_n,x_n)_{n \in \N}$ satisfying \eqref{sequences} there is a subsequence such that almost surely
\begin{align}\label{scalinglimit}
\frac{1}{r_{n_j}} \,(Q^{x_{n_j}}(r_{n_j}t), T^{x_{n_j}}(r_{n_j}t))\longrightarrow 
(\overline{Q}(t),\overline{T}(t)) \quad \text{ u.o.c.  as   } j\rightarrow \infty,
\end{align}
where $\overline{Q}(\cdot),\overline{T}(\cdot) \in C(\R_+,\R_+^K)$. For a fixed queueing discipline any limit $\overline{Q}(\cdot)$ is called a fluid limit path of the discipline with initial level $\overline{Q}(0)$, if
$(\overline{Q}(t),\,\overline{T}(t))$ are limits in the sense above. The set
of all such fluid limits $\overline{Q}$ is denoted by $\mathcal{Q}_{L}$. We
define the fluid limit model as the closure with respect to uniform
convergence on compact intervals of $\mathcal{Q}_{L}$ and denote it
by $\overline{\mathcal{Q}}_{L}$.

\begin{lemma}
The fluid limit model $\overline{\mathcal{Q}}_L$ is stable if and only if $\mathcal{Q}_{L}$ is stable.
\end{lemma}

\begin{proof}
    Obviously, if $\overline{\mathcal{Q}}_L$ is stable then
    $\mathcal{Q}_L$ is stable. Conversely, assume that $\mathcal{Q}_L$ is
    stable. Let $Q_*(\cdot)\in \overline{\mathcal{Q}}_L \, \backslash
    \mathcal{Q}_L$ and $Q_n(\cdot) \in \mathcal{Q}_L$ be a sequence such
    that $Q_n(\cdot) \rightarrow Q_*(\cdot)$ u.o.c. as $n\rightarrow
    \infty$. Since $\mathcal{Q}_L$ is stable there is a uniform $\tau>0$
    such that $Q_n(\tau+\cdot)\equiv 0$ for all $n \in \mathbb{N}$. It
    follows for all $t\geq \tau$ that
\begin{align*}
 Q_*(t) = \lim_{n\rightarrow \infty} Q_n(t) = 0
\end{align*}
and the proof is completed.

\end{proof}

\begin{proposition}
The fluid limit model $\overline{\mathcal{Q}}_L$ defines a closed GFN model.
\end{proposition}

\begin{proof}
The Lipschitz continuity and the scaling property are shown in \cite{Yechen}. 
To show the shift property we follow an idea that is due to \cite[Section~9.2.3]{robert2003queues}. Let $\mathcal{FL}(\overline{q})$ denote the set of fluid limits with initial level $\overline{q}$, i.e.
\begin{align*}
\mathcal{FL}\,(\overline{q}) := \{ \, \overline{Q}: \R_+  \rightarrow \R_+^K \, : \,  \overline{Q}(t,\overline{q}) = 
\lim_{ {n \rightarrow \infty,}}  \frac{ 1}{r_n}Q^{x_n}(r_n\,t)   \, , \, \overline{Q}(0)=\overline{q} \,  \} 
\end{align*}
We fix a pair of sequences $(r_n,x_n)_{n\in \mathbb{N}}$ that satisfies \eqref{sequences}
and $\lim_{n \rightarrow \infty} \frac{x_n}{r_n} = (\overline{q},0,0)$. Then, by the Skorokhod's Theorem~\cite[Theorem C.6]{robert2003queues} we have along a subsequence
\begin{align*}
\lim_{k\rightarrow \infty} \frac{1}{r_{n_k}} \,Q^{x_{n_k}}(r_{n_k}\,t) = \overline{Q}^{{n_k}}(t,\overline{q}) \in \mathcal{FL}(\overline{q}). 
\end{align*}
a.s. in the Skorokhod topology. The superscript to the fluid limit expresses the dependence on the particular sequence. Moreover, by the Markov property we have the following equality in distribution
\begin{align}\label{eq:eq-distr}
Q^{x_{n_k}}(r_{n_k}(t+s)) \overset{d}{=} Q^{ Q^{x_{n_k}}(r_{n_k}s) }(r_{n_k}t).
\end{align}
Also, by Proposition~3.5.2 in \cite{ethier1986markov} and $t \mapsto \overline{Q}^{{n_k}}(t,\overline{q})$ is continuous  
it holds that
\begin{align*}
\lim_{k\rightarrow \infty} \frac{1}{r_{n_k}} \,Q^{x_{n_k}}(r_{n_k}s) = \overline{Q}^{{n_k}}(s,\overline{q})\quad \mbox{ a.s.}. 
\end{align*}
Consequently, dividing \eqref{eq:eq-distr} by $r_{n_k}$ and taking limits yields that
\begin{align*}
\overline{Q}^{{n_k}}(t+s) \overset{d}{=} \overline{Q}^{n_k} (t, \overline{Q}^{{n_k}}(s))
\end{align*}
and hence we have
\begin{align*}
\overline{Q}^{n_k}(\cdot+s,\overline{q})  \in   \mathcal{FL} \, (\, \overline{Q}^{n_k}(s,\overline{q})\,  ) \,. 
\end{align*}
This shows the assertion.

\end{proof}

In the following we consider queueing networks under disciplines that are memoryless in the sense that the allocation process $\overline{T}$ of the fluid limit model at a time $t$ does only depend on the queue length at that time $t$. In particular, it does not require information of the past. In terms of the fluid limit models described in \cite{dai} this means that only the fluid level at a given time is needed to describe the evolution of the fluid level process.
Note that this explicitly excludes a number of disciplines as e.g. FIFO networks. We will comment on FIFO fluid networks in Section~\ref{FIFOsec}.
We also note that the problem of concatenating fluid limits was also
addressed by A. Stolyar \cite{stolyar95} and Ph. Robert
\cite[Section~9.2.3]{robert2003queues}. In \cite{stolyar95} it is shown that if the
queueing disciplines in every station satisfy a certain 'uniqueness
condition' on the disciplines of the individual servers the concatenation
property holds. However, there the definition of state is different,
because the state as used in \cite{stolyar95} includes the past trajectory
of the queue. Furthermore, in \cite{robert2003queues} concatenation is
possible if the fluid limits going through a certain queue level $Q$ are
unique.

\begin{remark}\label{rem:concat-not}
    Consider a queueing network with a memoryless discipline. We
    conjecture that in this case the fluid limit model
    $\overline{\mathcal{Q}}_L$ satisfies the concatenation
    property. Unfortunately, this claim has shown some resilience towards
    attempts of proof.
\end{remark}

Due to fact that $\overline{\mathcal{Q}}_L$ is closed by definition we would obtain the following result.
\vspace{.1125cm}\\
\textbf{Conjecture}\hspace{1ex} {\it The fluid limit model of a
  ''memoryless'' discipline defines a GFN model satisfying (e). It is
  stable if and only if it admits an upper semicontinuous Lyapunov
  function.}

\vspace{.25cm}
The conjecture holds true for the systems considered in
\cite{robert2003queues}, but unfortunately, the interesting fluid limits
do not have unique paths. As to the question of under which conditions
fluid limit models satisfy condition (f) we dare not venture a conjecture.

\section{The linear Skorokhod problem}\label{sec:skorokhod}

Another possible way to approximate a multiclass queueing network is to consider the so called diffusion limit. This limit can be regarded as a semi-martingale reflected Brownian motion (SRBM). Similar to the fluid limit, a sufficient condition for the stability of the SRBM is the stability of the linear Skorokhod problem (LSP) \cite{DP94}. The following description is taken from \cite{chen96} and \cite{Yechen}. Let $R$ be a $J\times J$ matrix, $\theta \in \R^J$ and $Z_0 \in \R_+^J$. The pair $(Z(\cdot),Y(\cdot)) \in C(\mathbb{R}_+,\R_+^J)$ is said to solve the LSP $(\theta,R)$ with initial state $Z_0$, if they jointly satisfy 
\begin{align}\label{LSP}
 Z(t) &= Z_0 + \theta t +R \,Y(t) \geq 0,\\ \label{LSP2}
{Y}(0) &=0 \text{  and  } {Y}(\cdot) \text{ is nondecreasing, } \\ \label{LSP3}
0&=\int_0^{\infty} Z_j(t) \,d{Y_j}(t), \qquad j=1,...,J.
\end{align}

The first question that arises is, which conditions guarantee the
existence of a solution of the LSP($\theta,R$). In oder to state such a
condition 
recall that 
a $J\times J$ matrix $R$ is said to be an $S$-matrix, if there
exists an $x\geq 0$ such that $Rx>0$, and is said to be completely-$S$ if
all of its principal submatrices are $S$-matrices. The following theorem
from \cite[Theorem 1]{bernard} contains the desired statement.

\begin{theorem}\label{sol-compl-S}
The LSP($\theta,R$) has a solution ($Z(\cdot),Y(\cdot)$) if and only if the matrix $R$ is completely-$S$.
\end{theorem}

Analogous to the previous subsections we define 
\begin{align*}
\mathcal{Q}_{LSP}=\{ Z(\cdot) : \, \exists \, Y(\cdot) \mbox{ such that } (Z(\cdot),Y(\cdot))\text{ satisfy }\eqref{LSP}-\eqref{LSP3}\}.
\end{align*}

Note that Theorem~\ref{sol-compl-S} states only the existence of a solution. In general the solution is not unique, for a counterexample see e.g. \cite{bernard}.

\begin{definition}\label{LSPstab}
A LSP$(\theta,R$) is said to be stable if, for any number $\e>0$ and any $Z(\cdot) \in \mathcal{Q}_{LSP}$ with $\|Z_0\|=1$, there exists a $\tau \geq 0$ such that $\|Z(\tau + \cdot)\| < \e$.
\end{definition}

To ensure that the set $\Q_{LSP}$ is nonempty, Theorem~\ref{sol-compl-S} states that $R$ has to be completely-$S$. In \cite[Theorem 5.2]{Yechen} it is shown that in this case Definition~\ref{LSPstab} is equivalent to Definition~\ref{GFNstab}. To derive a necessary and sufficient condition for stability of the linear Skorokhod problem we have to show that $\mathcal{Q}_{LSP}$ is a strict GFN model. The next lemma from \cite[Lemma 1]{bernard} or \cite[Lemma 5.1]{Yechen} implies that $\mathcal{Q}_{LSP}$ satisfies the Lipschitz condition.

\begin{lemma}\label{reglip}
If the matrix $R$ is completely-$S$, then there exists a constant $M$ such that any solution ($Z(\cdot),Y(\cdot)$) of LSP$(\theta,R$) is Lipschitz continuous with constant $M$. 
\end{lemma}

The fact that $\mathcal{Q}_{LSP}$ is closed follows from Proposition 1 in \cite{bernard}. Furthermore that the scale, shift property hold is stated in \cite[Section 2]{elkharroubi}. So it remains to investigate whether $\mathcal{Q}_{LSP}$ satisfies the concatenation and the lower semicontinuity property. Again we bring the linear Skorokhod problem into the context of differential inclusions. That is, let $\dot Y(t)= u$ and 

\begin{align}\label{LSP_DIGdef}
 G(Z) =  \left\{  \theta  +R  u \, : \,u \in U_{LSP}(Z) \right\},
\end{align}

where the set of admissible controls $U_{LSP}$ is determined through the conditions

\begin{align}\label{LSPu1}
u &\geq 0, \quad
Z_j \,u_j= 0, \quad \forall \, \, j=1,...,J.
\end{align}

While it is clear that the set described by \eqref{LSPu1} is unbounded on the boundary of the positive orthant, Lemma~\ref{reglip} may be used to see that the effective set of controls is bounded.
Indeed from the Lipschitz continuity of solutions, only values of $u$ below a certain bound need to be considered in \eqref{LSP_DIGdef}. 
 The corresponding differential inclusion is of the form

\begin{align}\label{LSP_DI}
\dot Z(t) \in G(Z(t)), \qquad Z(0)=Z_0.
\end{align}

It can can seen that the right-hand side is upper semicontinuous and the set $G(Z)$ is convex and compact. Again arguments from the theory of differential inclusions show the validity of the concatenation property. 

\begin{theorem}
$\mathcal{Q}_{LSP}$ is a closed GFN model satisfying (e). It is stable if and only if it admits an upper semicontinuous Lyapunov function.
\end{theorem}

We note that it is not obvious whether the right-hand side is also lower
semicontinuous.

\begin{remark}
    The consequence of the above theorem is, that the main theorem is
    applicable for the linear Skorokhod problem. However, for the provided
    Lyapunov function we can only show upper semicontinuity.
\end{remark}

\section{Remarks on fluid networks under FIFO service discipline}\label{FIFOsec}

For fluid networks that work under the FIFO service discipline the fluids are served in the order of their arrivals. To describe the evolution of class $k$ fluids we have to consider the workload $W(t) = C \, M^{-1}\, Q(t)$ of the station $j=s(k)$. For any time $t$ all jobs that arrive later than $t$ have lower priority in the FIFO discipline. So fluids that arrive at time $t$ are served at time $t +W_j(t)$.
The total arrivals of each fluid class until time $t$ is $A(t) = \alpha t + P^T \, M \,T(t)$. The characteristic of a FIFO fluid network can for each class $k \in \{1,...,K\}$ be represented by the following relation
\begin{align}\label{FIFO}
T_k(t + W_j(t)) = m_k (Q_{k}(0) + A_k(t)),
\end{align}
where $m_k= \mu_k^{-1}$. Note that the fluid network is not completely determined by the initial fluid level $Q(0)$ as it has to specified in which order the initial fluid level is served in the time period $[0,W_j(0)]$. So the initial data for each class $k \in \mathcal{K}$ is given by 
\begin{align*}
\{ T_k(s)\, : \,   s \in [0,W_j(0)] \, \}.
\end{align*}

The dynamics of a fluid network under FIFO service discipline is given by \eqref{D1}-\eqref{D4} and \eqref{FIFO}. 
Analogously to the previous disciplines we denote
\begin{align*}
\mathcal{Q}_F=\{ Q(\cdot) : \, \exists \, T(\cdot) \mbox{ such that } (Q(\cdot),T(\cdot)) \text{ is a solution }  \eqref{D1}-\eqref{D4},\eqref{FIFO}\, \}.
\end{align*}
In \cite[Lemma 3.7]{Yechen} it is shown that $\mathcal{Q}_F$ is a closed GFN model. However the fluid networks under FIFO discipline differ from the previous fluid models. One reason for this is the following. Consider again the flow balance equation in differential form, i.e.
\begin{align*}
 \dot Q (t) =  \alpha - (I-P^T)\, M \dot T(t).
\end{align*}
In the FIFO case the allocation process has to satisfy a functional differential equation of neutral type \cite{hale93}, since the allocation process has to satisfy the differential form of condition \eqref{FIFO} 
\begin{align*}
 \dot T_k (t + W_j(t)) \, \, (1+  \dot W_j(t)) =  m_k \alpha_k - m_k \,\sum_{l=1}^K p_{lk} \mu_l \dot T_l(t).
\end{align*}
The second reason is given two paths $Q_1(\cdot)$ and $Q_2(\cdot)$ of $\mathcal{Q}_F$ that coincide at some time, they will in general have different history, so that the concatenation is not immediately possible. In this context the initial data $\{ T_k(s)\, : \,   s \in [0,W_j(0)] \, \}$ plays a key role. 
An explicit counterexample to the concatenation property for FIFO networks
my be found in \cite{schoenleinthesis}.

\section{Conclusion}
\label{conclusions}

In this paper we have derived a converse Lyapunov theorem for generic fluid
networks under a concatenation condition. Continuity of the Lyapunov function
is ensured if the solution set of the fluid network has a a lower
semicontinuity property. Continuity is of interest because this would ensure
robustness properties of the network subject to unknown parameters or external
perturbations. The interesting class of FIFO networks does not immediately
fall under the results presented here. The question of a Lyapunov theory for
this and related cases is the subject of ongoing research.


%

\end{document}